\documentclass[10pt,letterpaper,oneside]{amsart}

  \usepackage{amsmath,amssymb,amsthm}
  \usepackage{amscd,mathtools}
  \usepackage{geometry}
  \usepackage{graphicx,epstopdf}
  \usepackage{color,xcolor}
  \usepackage{enumerate}
  \usepackage{xspace}
  \usepackage{tipa}
  \usepackage{epsfig}
  \usepackage{pinlabel,breqn}
  \usepackage[all]{xy}

  \usepackage{hyperref}

  \newcommand{\calC}{\mathcal{C}}

  \newcommand{\calL}{\mathcal{L}}

  \newcommand{\calP}{\mathcal{P}}

  \newcommand{\calX}{\mathcal{X}}

  \newcommand{\FF}{\mathbb{F}}

  \newcommand{\PP}{\mathbb{P}}
  
  \newcommand{\RR}{\mathbb{R}}

  \newcommand{\ZZ}{\mathbb{Z}}

 \newcommand{\btau}{{\boldsymbol \tau}}

  \newtheorem{theorem}{Theorem}[section]
\newtheorem*{diamtheorem}{Theorem~\ref{diam}}

  \newtheorem*{gentheorem}{Theorem~\ref{gen}}
\newtheorem*{atotheorem}{Theorem~\ref{ato}}
  \newtheorem*{twisttheorem}{Theorem~\ref{twist}}

  \newtheorem{proposition}[theorem]{Proposition}

  \newtheorem{lemma}[theorem]{Lemma}
  
  \newtheorem*{conjecture*}{Conjecture}

  \theoremstyle{definition}
  \newtheorem{definition}[theorem]{Definition}

  \newtheorem*{claim*}{Claim}

  \newtheorem*{question*}{Question}
  \newtheorem*{answer*}{Answer}
  \newtheorem*{application*}{Application}

  \theoremstyle{remark}
  
  \newtheorem*{remark*}{Remark}

  \DeclareMathOperator{\diam}{diam}

  \newcommand{\Out}{\ensuremath{\operatorname{Out}}\xspace}
  
  \newcommand{\Teich}{{Teichm\"uller }}

  \renewcommand{\a}{\alpha}
  
  \newcommand{\la}{\langle} 
  \newcommand{\ra}{\rangle}

\renewcommand{\a}{\alpha} 
  
\renewcommand{\t}{\tau}

  \newcommand{\param}{{\mathchoice{\mkern1mu\mbox{\raise2.2pt\hbox{$
  \centerdot$}}
  \mkern1mu}{\mkern1mu\mbox{\raise2.2pt\hbox{$\centerdot$}}\mkern1mu}{
  \mkern1.5mu\centerdot\mkern1.5mu}{\mkern1.5mu\centerdot\mkern1.5mu}}}

  \newcommand{\ttau}{{\widetilde \tau}}

\newcommand{\bS}{{\boldsymbol S}}
\newcommand{\ba}{{\boldsymbol a}}

\begin{document}


\title {Fully irreducible Automorphisms of the Free Group\\
 via Dehn twisting
in $\sharp_k(S^2 \times S^1)$}

  \author   {Funda G\"ultepe}
\address{Department of Mathematics\\
University of Illinois at Urbana-Champaign\\ Urbana, IL 61801}
\email{fgultepe@illinois.edu}
\urladdr{http://www.math.illinois.edu/~fgultepe/}

\begin{abstract} By using a notion of a geometric Dehn twist in $\sharp_k(S^2 \times S^1)$, we prove that when projections of two $\ZZ$-splittings to the free factor complex are far enough from each other in the free factor complex, Dehn twist automorphisms corresponding to the $\ZZ$ splittings generate a free group of rank $2$. Moreover, every element from this free group is either conjugate to a power of one of the Dehn twists or it is a fully irreducible outer automorphism of the free group. We also prove that, when the projections of $\ZZ$--splittings are sufficiently far away from each other in the intersection graph, the group generated by the Dehn twists have automorphisms either that are conjugate to Dehn twists or are atoroidal fully irreducibles.
\end{abstract}

\maketitle
\section{introduction}

Due to their dynamical properties, \textit{fully irreducible} outer automorphisms are important to understand the dynamics and geometric structure of $\Out(F_k)$ and its subgroups.(\cite{LevLust1}, \cite{CLAYP}, \cite{bestclay1}). Just like pseudo Anosov surface homeomorphisms, fully irreducibles are characterized to be the class of automorphisms no power of which fixes a conjugacy class of a proper free factor of $F_k$. Since their dynamical properties and their role in $\Out(F_k)$ is similar to those of pseudo Anosov mapping classes for the mapping class group, to construct fully irreducibles it is natural to seek ways similar to those of pseudo Anosov constructions. In this work, we will provide such construction using \textit{Dehn twist automorphisms}, by composing powers of Dehn twists from the free group of rank $2$ that they generate. This is inspired by the work of Thurston on pseudo Anosov mapping classes of mapping class group of a surface (\cite{Thurston1}) yet we use the similar ping pong methods Hamidi-Tehrani uses in his generalization of Thurston's to Dehn twists along multicurves (\cite{HTEH02}).

Constructing fully irreducible automorphisms by composing certain (possibly powers of) other automorphisms and locating free groups is not new to the study of automorphisms of the free group. For instance, Clay and Pettet in \cite{CLAYP} constructed fully irreducibles by composing elements of a free group of rank $2$ which was generated by powers of two Dehn twist automorphisms. However, the powers of the Dehn twists used to generate the free group were not uniform but depended on the twists, one needed to take a different power for each pair of Dehn twists to obtain a free group.

In their work Clay and Pettet studied Dehn twists algebraically, as outer automorphisms of the free group and they used algebraic tools to study them. As a result their construction produced the nonuniform powers of twists. In this paper, our goal is constructing fully irreducible automorphisms by studying Dehn twist automorphisms differently. We first change the model for $\Out(F_k)$ from the $1$--dimensional one to the $3$--dimensional one, $M=\sharp_k(S^2 \times S^1)$. This way, we were able to understand Dehn twists geometrically, using essential imbedded tori in $M$. This approach resulted in a more geometric construction of fully irreducibles.

More specifically, we will prove the following theorem using \textit{geometric Dehn twists}.

\begin{theorem}\label{twist} 
Let $T_1$ and $T_2$ be two $\mathbb{Z}$--splittings of the free group $F_k$ with rank $k>2$ and $\alpha_1$ and $\a_2$ be two corresponding free factors in the free factor complex $FF_k$ of the free group $F_k$. Let $D_{1}$ be a Dehn twist fixing $\alpha_1$ and $D_{2}$ a Dehn twist fixing $\alpha_2$, corresponding to $T_1$ and $T_2$, respectively. Then there exists a constant $N=N(k)$ such that whenever $d_{FF_k}(\alpha_1, \a_2) \geq N$,
\begin{enumerate}
\item $\langle D_{1}, D_{2}\rangle \simeq F_2$.
\item All elements of $\langle D_{1}, D_{2}\rangle$ which are not conjugate to the powers of $D_{1}, D_{2}$ in $\langle D_{1}, D_{2}\rangle$  are fully irreducible.
\end{enumerate}

\end{theorem}
Now we would like to give the definitions necessary to understand the statement of the Theorem \ref{twist} and explain the ideas used in its proof.

\medskip \noindent {\bf Splittings and $\Out(F_k)$ Complexes:}

 A \textit{Dehn twist automorphism} is an element of $\Out(F_k)$ defined by using $\mathbb{Z}$--splittings of $F_k$ either as an amalgamated free product $F_k=A\ast_{\la c \ra}B$ or as an HNN extension of the free group $F_k= A\ast_{\la c \ra}$. More precisely, it is induced by the following automorphisms in each case:
\begin{align*}
  A*_{\la c \ra} B:& \, a \mapsto a \text{ for } a \in A & \quad
  A*_{\la tct^{-1} = c' \ra}:& \, a \mapsto a \text{ for } a \in A \\
  & \, b \mapsto cbc^{-1} \text{ for } b \in B & & \, t \mapsto tc
\end{align*}
Given a $\mathbb{Z}$--splitting of $F_k$ as $F_k=A_1\ast_{\la c_1 \ra}B_1$ at least one of $A_1$, $B_1$ is a proper free factor. In HNN extension case $F_k= A_1\ast_{\la c_1 \ra}$ the stable letter is a proper free factor. By Bass--Serre theory, each $\mathbb{Z}$--splitting of $F_k$ gives rise to a tree whose quotient with respect to the action of the free group is a single edge. The edge stabilizer is $\mathbb{Z}$ and  vertex stabilizers in the amalgamated case are $A_1$ and $B_1$ while in the HNN case it is $A_1$. We will coarsely \textit{project} each splitting onto the vertex which is a proper free factor. In the amalgamated case we consider the Dehn twist automorphism corresponding to the $\mathbb{Z}$--splitting which fixes this free factor and in the HNN case the Dehn twist automorphism will be the one fixing the vertex stabilizer. We study the action of this Dehn twist on the \textit{free factor complex} $FF_k$ of the free group $F_k$ of rank $k$ and we determine under which conditions the compositions of the Dehn twist automorphisms give fully irreducible automorphisms. The free factor complex is a simplicial complex whose vertices are conjugacy classes of proper free factors of $F_k$ and the adjacency between two vertices corresponding to two free factors $A$ and $B$ is given whenever $A < B$ or $B< A$. This complex was first introduced by Hatcher and Vogtmann in \cite{H3} as a \textit{curve complex} analog for $\Out(F_k)$ and in this work we will use its geometric properties due to its hyperbolicity, which is given in \cite{BestHyp} by Bestvina and Handel.

There are several geometrically distinct hyperbolic simplicial complices $\Out(F_k)$ acts on by simplicial automorphisms which are considered to be analogs to the curve complex for the mapping class group. Contrary to the case with the curve complex and the action of the mapping class group on it, it is not always possible to identify fully irreducible elements with respect to the way they act on a curve complex analog. For example, an element of $\Out(F_k)$ might act hyperbolically on a curve complex analog yet it may not be fully irreducible. In this work the free factor complex was used since loxodromic action of an automorphism on free factor complex completely characterizes being fully irreducible for a free group automorphism. Thus, to identify fully irreducibles in a group generated by two Dehn twists, it is enough to have a loxodromic action.  

As \textit{geometric} Dehn twist we mean the following. For each equivalence class of a $\ZZ$--splitting, by Lemma \ref{CR}, there is an associated homotopy class of a torus in $M$. More specifically, an amalgamated free product gives a separating torus in $M$ whereas an HNN extension corresponds to a non-separating torus. Hence, each Dehn twist automorphism corresponds to a Dehn twist along the torus given by the $\ZZ$--splitting. The Dehn twist along a torus will be called a geometric Dehn twist.

\medskip \noindent {\bf Dehn twists and their almost fixed sets:}
To prove the main theorem we use a \textit{ping pong argument for elliptic type subgroup} since Dehn twists have fixed points in the free factor complex. To define such argument one needs to construct so called ping pong sets. Thus we need to know first that the points of the free factor complex which are not moved too far away by a power of a Dehn twist are manageable. More precisely, let $\phi\in \Out(F_k)$ and let
\[F_C(\phi)= \{x \in FF_k: \exists n \neq 0\,\,\text{such that}\,\, d(x, \phi^n(x)) \leq C \}
 \]
be its almost fixed set in $FF_k$.
The following theorem is the main ingredient in the elliptic type ping pong argument.

\begin{theorem} \label{diam} Let $T$ be a $\mathbb{Z}$--splitting of the free group $F_k$ with $k >2$,  and $D_{T}$ denote a corresponding Dehn twist. Then, for all sufficiently large constants $C$, there exists a $C'=C'(C,k)$ such that the diameter of the almost fixed set
\[
  F_C= \{x \in FF_k: \exists n \neq 0\,\, \text{such that}\,\,d(x, D_{T}^n(x)) \leq C \}
\]
corresponding to $ \langle D_{T} \rangle$ is bounded above by $C'$.
\end{theorem}

\medskip \noindent {\bf Relative twisting and distances along paths:}
Now, to prove that the almost fixed sets of Dehn twists have bounded diameter, one needs to understand distances between points in the free factor complex. However we cannot assume that there is a geodesic between two points in the free factor complex which is appropriate for our calculation purposes since we do not know what these geodesics are. But it is known that the \textit{folding paths} in outer space give rise to geodesics in outer space and their projections to free factor complex are quasigeodesics. To prove the Theorem \ref{diam}, we have proved that there is a folding path whose projection to the free factor complex is at a bounded distance from the given free factor. To achieve this one would need an analog of the \textit{annulus projection} and to be able to calculate distances on an annulus complex. Then using a version of \textit{Bounded Geodesic Image Theorem} of \cite{MM2} one would conclude that whenever the number of twists is more than the universal constant given in this theorem, the quasi geodesic between a point and its twisted image has a vertex which does not intersect the core curve of the annulus. However, we do not have the main tool for the purpose, which is an analog for annulus projection, since the subfactor projection is not defined for free factors of rank one (\cite{BestSF}, \cite{ST1}.)

To calculate distances without using a projection  we refer to a  theorem of Clay and Pettet in \cite{CP2}
in which they give a pairing $tw_a(G,G{'})$ named \textit{relative twisting number}
between two graphs $G, G{'} \in CV_k$ relative to some nontrivial $a\in F_k$,
which is defined by means of the Guirardel core. 
Using this pairing, they obtain a condition on
the graphs $G, G{'} \in CV_k$ that, when satisfied, enables them construct a connecting
geodesic between them, traveling through thin part of $CV_k$.

\medskip \noindent {\bf Relative twisting along tori in $\displaystyle{\sharp_k(S^2 \times S^1)}$:}
We have used the interpretation of \textit{the relative twisting number} pairing $tw_a(G,G{'})$ for two spheres relative to an element of the free group, which is in our case the generator of the core (longitude) curve of a torus. Then the relative twist is a number which calculates distances between two spheres which are intersecting the same torus along its core curve. Mimicking annulus projection, the relative twisting number might be interpreted as number of intersections between \textit{projections} of some spheres in $M$ onto a torus (yet we do not make a formal definition of such a projection). With relative twisting number we were able to calculate a lower bound for the twisting number between a sphere and its Dehn twisted image, relative to a torus hence in some sense relative to a rank-$1$ free factor (related to its core curve). Afterwards, a lemma of \cite{CP2} by Clay and Pettet guarantees the existence of a geodesic between the corresponding points in the outer space along which core curve gets short. Using a Bestvina-Feighn lemma (\cite{BestHyp}) we project this geodesic to the free factor complex and using the distance calculations we show easily that the almost fixed set of a Dehn twist automorphism has a bounded diameter. This completes the preparation for ping pong with elliptic type group as it is given by Kapovich and Weidmann in \cite{KWeid1}. Now we have ping pong sets that we have control over.

The main argument which also finishes the proof of our main result is encoded in the following theorem.
\begin{theorem}\label{gen}
Let $G$ be a group acting on a $\delta$-hyperbolic metric space $X$ by isometries and $\phi_1, \phi_2 \in G$. Suppose $C>100 \delta$ and the almost fixed sets  $\calX_C(\phi_1)$ and $\calX_C(\phi_2)$ of $\langle \phi_1 \rangle$ and $ \langle \phi_2 \rangle$ respectively have diameters bounded above by a constant $C{'}$. Then there exists a constant $C_1$ such that, whenever  $d_X(\calX_C(\phi_1),\calX_C(\phi_2)) \geq C_1$,
\begin{enumerate}
\item $\langle \phi_1, \phi_2 \rangle \simeq F_2$ and,

\item every element of $\langle {\phi_1}, {\phi_2}\rangle $ which is not conjugate to the powers of $\phi_1, \phi_2$ in $\langle {\phi_1}, {\phi_2}\rangle $ acts loxodromically in $X$.
\end{enumerate}

\end{theorem}

Finally, we \textit{project} Dehn twists to \textit{intersection graph} $\calP_k$. Its vertex set consists of marked roses up to equivalence and there are two types of edges between vertices. The first is that whenever two roses share an edge with the same label, corresponding vertices are connected by an edge in $\calP_k$. Second one is obtained whenever there is a marked surface with one boundary component such that the element of the fundamental group represented by the boundary crosses each edge of both roses twice. This simplicial complex is closely related to the \textit{intersection graph} introduced by Kapovich and Lustig (\cite{KaLu2}) and it is hyperbolic. (\cite{MannTh}).

A fully irreducible automorphism is called \textit{geometric} if it is induced by a pseudo-Anosov homeomorphism of a surface with one boundary component. A fully irreducible automorphism is \textit{atoroidal} if no positive power of it preserves the conjugacy class of a nontrivial element of $F_k$. Moreover, only non-geometric fully irreducible automorphisms are atoroidal by a theorem of Bestvina-Handel (\cite{BH1}). The important feature of the intersection graph for us is that the atoroidal fully irreducibles act loxodromically on this graph (\cite{MannTh}).

We obtain the following theorem.

\begin{theorem}\label{ato} Let $T_1$ and $T_2$ be two $\mathbb{Z}$--splittings of $F_k$ with $k>2$ with corresponding free factors $\alpha_1$ and $\alpha_2$ and let  $D_{1}$ and $D_{2}$ be two Dehn twists corresponding to $T_1$ and $T_2$, respectively.  Then there exists a constant $N_2=N_2(k)$ such that whenever $d_{\calP_k}(\sigma(\alpha_1), \sigma(\alpha_2)) \geq N_2$ , $\langle D_{_1}, D_{_2}\rangle \simeq F_2$ and all elements from this group which are not conjugate to the powers of $D_{1}, D_{2}$ in $\langle D_{1}, D_{2}\rangle$  are atoroidal fully irreducible.
\end{theorem}

\subsection*{Acknowledgements} I would like to thank Ilya Kapovich for asking the question which led to the writing of this paper. I am especially indebted to Chris Leininger for his constant support and for suggesting a shorter proof for the main theorem. I am very grateful for Matt Clay for giving me permission to use his proof in this paper. Many thanks also to Mark C. Bell for reading the preprint and making many useful suggestions. Lastly I would like to specifically thank the anonymous referee whose suggestions and corrections improved the paper substantially.
\section{Preliminaries}

\subsection{Sphere systems and normal tori}

Let $M=\sharp_k(S^2 \times S^1)$. Then $\Out(F_k)$ is isomorphic to the mapping class group $MCG(M)$ of $M$ up to twists about $2$-spheres in $M$ (\cite{LAUD}). $M$ can be described as follows: we remove the interiors of $2k$ disjoint 3-balls from $S^3$ and identify the resulting 2-sphere boundary components in pairs by orientation-reversing diffeomorphisms, creating $k$ many $S^2\times S^1$ summands.

 Associated to $M$ is a rich algebraic structure coming from the essential $2$-spheres that $M$ contains. A \textit{sphere system} is a collection of isotopy classes of disjoint and non-trivial 2-spheres in $M$ no two of which are isotopic. 


We call a collection $\Sigma$ of disjointly imbedded essential, non-isotopic $2$-spheres in $M$
\textit{a maximal sphere system} if every  complementary component of $\Sigma$ in $M$ is a $3$-punctured $3$-sphere.

A fixed maximal sphere system $\Sigma$ in $M$ gives a description of the universal cover $\widetilde{M}$ of $M$ as follows. Let $\PP$ be the set of  $3$-punctured $3$-spheres in $M$ given by a maximal sphere system $\Sigma$ and regard $M$ as obtained from copies of $P$ in $\PP$ by identifying pairs of boundary spheres. Note that a pair might both be contained in a single $P$, in which case the image of $P$ in $M$ is a once-punctured $S^2\times S^1$. To construct $\widetilde{M}$, begin with a single copy of $P$ and attach copies of the $P$ in $\PP$ inductively along boundary spheres, as determined by unique path lifting. Repeating this process gives a description of $\widetilde{M}$ as a treelike union of copies of the $P$. We remark that $\widetilde{M}$ is homeomorphic to the complement of a Cantor set in $S^3$.

To be able to define a concept of Dehn twist, we need to use the correspondence between the equivalence classes of $\mathbb{Z}$--splittings of $F_k$ and homotopy classes of \textit{essential} tori in $M$, which is given in Lemma \ref{CR}.

For us, a torus in $M$ is an imbedded torus in $M$ so that the image of the fundamental group of torus is a cyclic group isomorphic to $\ZZ$ in $\pi_1(M)$. Moreover, we  consider \textit{essential} tori only, the tori which do not bound a solid torus in $M$. There are two types of essential tori in $M$, depending on the type of the splitting of the free group they correspond to. Namely, for an amalgamated free product we have a separating torus in $M$ and a non-separating one for an HNN extension of the free group. Two examples could be seen in the Figure \ref{fig:MT}.
\begin{figure}
\centering{\includegraphics[scale=0.50]{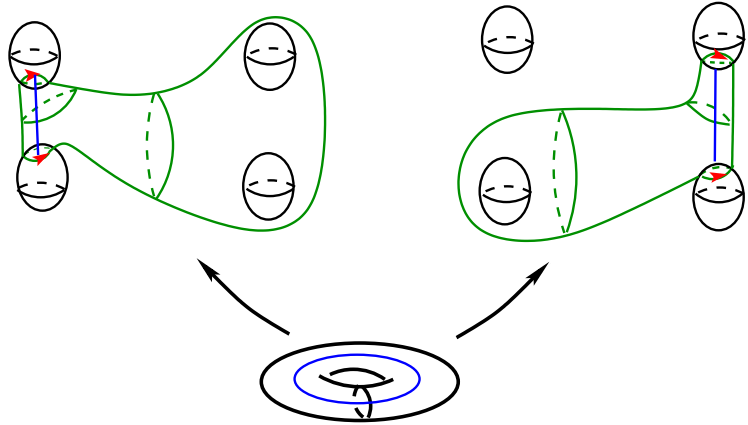}}
\caption{Imbeddings with the given identifications correspond to a separating (on the left) and a non-separating (on the right) tori in $\sharp_4 (S^2\times S^1)$.}
\label{fig:MT}
\end{figure}

As an analog of a geodesic representative from a homotopy class of a curve, among the representatives of a torus in a homotopy class, it is necessary to identify one which intersects the spheres in a maximal sphere system of $M$ minimally. To this end, the \textit{normal form} for tori is defined. Following Hatcher's normal form for sphere systems in \cite{H1}, a normal form for tori defined in \cite{F1} so that the each intersection of the torus with each complementary $3$-punctured $3$-sphere is a disk, a cylinder or a pants piece. And according to \cite{F1}, if a torus $\tau$ is in \textit {normal form} with respect to a maximal sphere system $\Sigma$, then the intersection number of $\tau$ with any $S$ in $\Sigma$ is minimal among the representatives of the homotopy class $\btau$. Here, by \textit{intersection number} we mean the number of components of intersection between the spheres of a maximal sphere system and the torus.

In this work, we will implicitly use the following existence theorem form \cite {F1}.

\begin{theorem}\cite{F1}
Every imbedded essential torus in $M$ is homotopic to a normal torus and the homotopy process does not increase the intersection number with any sphere of a given maximal sphere system $\Sigma$.
\end{theorem}

\subsection{Geometric models, complexes and projections}

Given the  free group $F_k$ on $k$ letters, the associated space of the marked metric graphs $CV_k$ which are homotopy equivalent to $F_k$ with total volume is $1$ introduced first by Culler and Vogtmann in \cite{Vogt1}. A marked metric graph is an equivalence class of pair of a metric graph $\Gamma$ and a marking, which is a homotopy equivalence with a rose. Outer space might be thought as an analogue to \Teich space for the mapping class group. For the details we refer the reader to \cite{Vogt1} and \cite{VogtSurvey}.

The \textit{free factor complex} of a free group is defined first by Hatcher and Vogtmann in \cite{H3} as a simplicial complex whose vertices are conjugacy classes of proper free factors and adjacency is determined by inclusion. It is hyperbolic by \cite{BestHyp}.

We will use the coarse projection $\pi: CV_k\rightarrow FF_k$ defined as follows: for each proper subgraph $\Gamma_0$ of a marked graph $G$ that contains a circle, we take its image in $FF_k$ as the conjugacy class of the smallest free factor containing $\Gamma_0$. Now by \cite{BestHyp}, for two such proper subgraphs $\Gamma_1$ and $\Gamma_2$, $d_{FF_k}(\pi(\Gamma_1),\pi(\Gamma_2) )\leq 4$ hence we have a coarsely defined map.

Another hyperbolic $\Out(F_k)$-- graph we refer to is the \textit{free bases graph} $FB_k$ given by Kapovich-Rafi (\cite{KaRa}).  For $k\geq 3$, this graph has vertices the free bases of $F_k$ up to equivalence, (two bases are equivalent if their Cayley graphs are $F_k$-equivariantly isometric) and whenever two bases representing the vertices have a common element these vertices are connected by an edge. What is useful for us is that $FB_k$ and the $FF_k$ are quasi isometric (\cite{KaRa}).

\textit{The intersection graph} $\calP_k$ has vertex set consisting of marked roses up to equivalence. Two roses are connected by an edge if either they have a common edge with the same label, or there is a marked surface with one boundary component and the representative of this component crosses each edge of both roses twice. There is a Lipschitz map between $FB_k$ and $\calP_k$, constructed by thinking of each basis of $F_k$ as its corresponding rose marking and observing that $\calP_k$ shares the edges and vertices of $FB_k$ and has some additional edges between roses.

\subsection{Tori in $M$ and $\ZZ$--splittings of the free group}
\vfill
In this section we will establish the correspondence between an equivalence class of a $\ZZ$--splitting and  a homotopy class of a torus in $M$.
Consider an imbedded essential torus $\tau$ in $M$. There is a simplicial tree associated to this torus.
To obtain this simplicial tree we take a neighborhood of each lift in the set of lifts  $\widetilde{\tau}$  of $\tau$ and we take a vertex for each complementary component. Two complementary components are adjacent if they bound the neighborhood of the same lift. We will denote this tree by $T_{\btau}$ and as correspondence between this tree and the torus $\tau$  we will understand the $F_k$--equivariant map $\widetilde{M} \rightarrow T_{\tau}$ which sends each complementary component of a neighborhood of a lift to a vertex and shrinks each such neighborhood to an edge. The tree constructed this way is referred to as the \textit{Bass-Serre} tree corresponding to $\tau$. Recall that by Bass--Serre theory, the action of $F_k$ on $T_\t$ gives a single--edged graph of groups decomposition of $F_k$, and hence a $\ZZ$--splitting of the free group $F_k$.

The next lemma gives the existence of an equivalence class of a $\ZZ$--splitting for each homotopy class of a torus in $M$ and its proof is based on the notion of the \textit{ends} of  $\widetilde{M}$.

An end of a topological space is a point of the so called \textit{Freudenthal compactification} of the space. More precisely,
\begin{definition} Let $X$ be a topological space. For a compact set $K$, let $C(K)$ denote the set of components of complement $X-K$. For $L$ compact with $K\subset L$, we have a natural map $C(L) \to C(K)$. These compact sets define a \textit{directed system} under inclusion. The set of ends $E(X)$ of $X$ defined to be the inverse limit of the sets $C(K)$.
\end{definition}

The space $\widetilde{M}$ is non-compact and it has infinitely many ends. We denote the set of ends of $\widetilde{M}$ by $E(\widetilde{M})$. It is homeomorphic to a Cantor set, in particular, it is compact. For a maximal sphere system $\Sigma$ in $M$, the set $E(T_{\Sigma})$ of ends of the Bass-Serre tree  $T_{\Sigma}$ of $\Sigma$ is identified with the set  $E(\widetilde{M})$. By analyzing this set we were able to prove the following.

\begin{lemma}\label{tree} Let $T_{\t_1}$ and $T_{\t_2}$ be two Bass-Serre trees corresponding to two tori $\t_1$ and $\t_2$, respectively. If $\t_1$ and $\t_2$ are homotopic, $T_{\t_1}= T_{\t_2}$ and hence $\t_1$ and $\t_2$ have equivalent $\mathbb{Z}$--splittings.
\end{lemma}

\begin{proof}
Let $\t$ be an embedded essential torus. We claim that each lift of $\t$ is two sided. For if it is not 2 sided, then there is a non-trivial (non-homotopic to a point) loop which intersects a lift once and connects an end of $\widetilde{M}$ to itself. Now by the loop theorem, this loop bounds a disk in $\widetilde{M}$. Then, the projection of this disk to $M$ bounds a disk in $M$, which means that the torus bounds a solid torus in $M$. This contradicts the fact that $\tau$ is essential.

 Hence each lift divides $\widetilde{M}$ into two disjoint  parts. A transverse orientation on the torus gives an orientation on the spheres on each such parts and hence there is  a labeling on the valence-$2$ vertices corresponding to these spheres. It is clear that each lift $\sf L = S^1
\times \RR$ of $\t$ defines a decomposition of the set of ends of $T_{\Sigma}$ into two sets $\sf L^+$
and $\sf L^-$ where $\sf L^+\cap \sf L^{-}$ consists of two endpoints, corresponding to the axis of the lift $\sf L$.  For each torus, let us consider all the endpoints corresponding to the axes of all lifts and eliminate them from  the set of ends $E(\widetilde{M})$ of $\widetilde{M}$. Let us denote the  remaining set  $\tilde{E}(\widetilde{M})$.

Now, for each lift $L$, we have a partition ($\sf L^+, \sf L^{-}$) of the set $\tilde{E}(\widetilde{M})$. Since the lifts are disjoint, any two lifts $L_1$ and $L_2$  satisfy either
$\sf L^+_1 \subset \sf L^+_2$ or $ \sf L^+_1\subset \sf L^-_2$.

We construct a tree corresponding to the  set of partitions as follows: Given a  pair of set of partitions ($\sf L_1^+, \sf L_1^{-}$) ($\sf L^+_2, \sf L_2^{-}$) with $ \sf L^+_1 \subset \sf L^+_2$ or $\sf L^+_1\subset \sf L^-_2$, whenever there is no collection of ends $(\sf Z^+,\sf Z^-)$ satisfying $ \sf L^+_1 \subset \sf Z^+\subset \sf L^+_2$ or $\sf L^+_1\subset \sf Z^-\subset \sf L^-_2$ we take an edge between them. For each maximal subset of   $\tilde{E}(\widetilde{M})$ which is not separated by any lift, we take a vertex.  Since the partitions of ends do not intersect, we have a tree.

Now, since for each lift we have a partition of the ends, there is an isomorphism between the tree
given by the partitions and the Bass Serre tree  $T_{\t}$ corresponding to $\t$. More precisely, we define a map between two trees by
taking the ``edge-midpoint'' vertices of  $T_{\t}$  to the set of partitions (i.e. the components of $\widetilde{\t}$).
and the components of $\widetilde{M} - \widetilde{\t}$, i.e.
vertices of the  $T_{\t}$, are mapped to the collections of lifts
(topologically, the frontier components of these components), having
the property that if $L_1$ and $L_2$ are two of them, then (assuming that
we select the notation so that $ \sf L^+_1 \subset \sf L^+_2$ there is no $L_3$
in the collection for which $ \sf L^+_1 \subset \sf L^+_3\subset \sf L^+_2$ or
$ \sf L^+_1 \subset \sf L^-_3 \subset \sf L^+_2$

Now we claim that for a homotopy of embedded tori in $M$, the initial and final tori
determine the same partition of the ends in  $\tilde{E}(\widetilde{M})$, and hence they have the same partition tree, and the same Bass Serre tree as a result.

To see this, let $\t_1$ be homotopic to $\t_2$. To show that we get the
same partitions of the endpoints from the lifts of $\t_1$ and from the
lifts of $\t_2$, we need to show that if two endpoints are separated
by a component $L$ of $\widetilde{\t_1}$, and $L$ is homotopic to $L'\in \widetilde{\t_2}$ then they are separated by $L^\prime$ too.

Let $p$ and $q$ be two endpoints separated by $L$. Fix an arc between them
that crosses $L = S^1 \times \RR$ in one point. During the homotopy, although no longer imbedded, $L$ moves in $\widetilde{M}$,
in particular, it does not touch any endpoint. So assuming that the homotopy is
transverse to the arc, its inverse image in $S^1\times \RR\times I$
consists of circles and arcs properly imbedded in $S^1 \times \RR\times I$.
Note that if the homotopy could cross an endpoint of the arc, then an
arc of the inverse image could fail to be properly imbedded in
$S^1\times \RR \times I$. But this does not happen since the homotopy between the two tori induces a homotopy between normal representatives of each tori, corresponding to a fixed maximal sphere system in $M$. By \cite{F1} such a homotopy is \textit{normal} at each stage hence cannot cross an endpoint.

Back to the inverse image of the homotopy between two tori, since only one endpoint of the inverse image of the arc is
in $L$, there must be an odd number of endpoints in $L^\prime$ (i. e. the arc
crosses $L^\prime$ an odd number of times) and therefore $L^\prime$ also separates $p$
and $q$.
\end{proof}

Recall that given a $\mathbb{Z}$--splitting of $F_k$, an associated \textit{Dehn twist automorphism} of $F_k$ is defined in two following ways.

\begin{align*}
  A*_{\la c \ra} B:& \, a \mapsto a \text{ for } a \in A & \quad
  A*_{\la tct^{-1} = c' \ra}:& \, a \mapsto a \text{ for } a \in A \\
  & \, b \mapsto cbc^{-1} \text{ for } b \in B & & \, t \mapsto tc
\end{align*}
On the left is the definition when the $\mathbb{Z}$--splitting is given by an amalgamated product $F_k=A\ast_{\la c \ra}B$ and on the right is the definition when the $\mathbb{Z}$ splitting is an HNN extension  $A\ast_{\la c \ra}$ of the free group $F_k$. Note that the Dehn twist automorphism in the amalgamated case is defined up to conjugacy since it is possible to reverse the roles of $A$ and $B$.

Before we give the last lemma in this section, we give two theorems relating $\ZZ$--splittings to free splittings which  will be used in the proof.
\begin{theorem}[Shenitzer \cite{shenitzer}]\label{she}
Suppose that a free group $\mathbb{F}$ is an amalgamated free product $\FF=A*_BC$ where $B$ is cyclic. Then $B$ is a free factor of $A$ or a free factor of $C$.
\end{theorem}
\begin{theorem}[Swarup \cite{swarup}] \label{swarup} Suppose that a free group $F$ is an HNN entension $\FF=A*_B $ where $B\neq 1$ is cyclic. Then $A$ has a free product structure $A=A_0*A_1$ in such a way that one of the following symmetric alternatives hold.
\begin{enumerate}
\item $B\subset A_0$ and there exists $a\in A$ such that $t^{-1} B t= a^{-1} A_1 a$. Or,
\item $t^{-1} B t \subset A_0$ and there exists $a \in A$ such that $B=a^{-1} A_1 a$.
\end{enumerate}
\end{theorem}
where $t$ is the stable letter.

Finally, the following lemma gives the converse relation between a torus and $\ZZ$--splitting hence explains why we are interested in tori in $M$. The argument is due to Matt Clay.

\begin{lemma}\label{CR} Given a $\ZZ$--splitting $Z$ and an associated Dehn twist automorphism, there is a torus $\tau$ in $M$ unique up to homotopy such that $T_Z= T_\t$ where $T_Z$ and $T_\t$ are the corresponding Bass-Serre trees.
\end{lemma}

\begin{proof} In this proof we will build a homotopy class of a torus from a sphere and a loop. First we use theorems \ref{she} and
 \ref{swarup} that relate a $\ZZ$--splitting of $F_k$ to a free splitting of $F_k$. Then to relate the free splitting to a homotopy class of  a sphere we use a theorem originally due to Kneser (\cite{kneser}). This theorem is later developed by Grushko \cite{Grus}, and most recently by Stallings (\cite{STA1}) and these are the versions we will be referring to. We treat the amalgamated product and
  HNN-extension cases separately. Amalgamated case has schematic pictures Figure \ref{fig:slide} and Figure \ref{fig:slide2} associated to the proof.

  \medskip \noindent {\bf Case 1:} We first consider the case of an
  amalgamated free product $F_k = A*_{\la b \ra}B$.  By Shenitzer's
  Theorem \ref{she}, ${\la b \ra}$ is either a free factor of $A$ or a free factor of $B$. Hence, there is a free splitting $F = A*B_0$ where $B =
  \la b \ra* B_0$ or $F_k = A_0*B$ with $A=A_0*\la b \ra$. Let us assume the former, and let $S \subset M$ be an imbedded (separating)
  sphere representing this splitting.  We fix a basepoint $* \in M$
  and assume it lies on $S$.  As $b \in A$, there is an embedded loop
  $\gamma \subset M$ that represents $b \in F$ and only intersects $S$
  at $*$.  For small $\epsilon$, boundary of the closed
  $\epsilon$--neighborhood of $S \cup \gamma$ consists of two
  components: an imbedded sphere isotopic to $S$ and an imbedded
  essential torus $\tau_\gamma$.

 Every torus can be written as a sphere and a loop attached to it. Hence it is clear from the construction, that the splitting of $F$
  associated to $\tau_\gamma$ is the original
  splitting.  However, there are some choices made in the construction
  of $\tau_\gamma$ and it must be shown that different choices result
  in homotopic tori.  It is clear that changing $S$ or $\gamma$ in the construction by
  a homotopy results in a change of $\tau_\gamma$ by a homotopy.

  Now since Shenitzer's theorem \cite{shenitzer} gives many possible splittings differed by automorphisms of $B$ fixing $\la b \ra$ (Hence Nielsen automorphisms that fix $b$), we need to
  consider two different complementary free factors $B_0$ and $B_1$ of
  $A$ such that $\la b \ra *B_0 = \la b \ra *B_1 = B$ and show that
  the tori obtained after we add the loop to corresponding spheres are homotopic, even when the spheres themselves are not.
  For this, let $S_0$ and $S_1$ be the spheres representing the splittings $A*B_0$ and $A*B_1$
  respectively and $\tau_0$ and $\tau_1$ be the tori as constructed
  above using these spheres.  We assume that $\gamma$ intersects $S_0$
  only at the fix basepoint $* \in M$.

  We first treat the special case that $B_1$ is obtained from $B_0$ by
  replacing a generator $x \in B_0$ by $xb$.  Fix a basis for $F$
  consisting of a basis for $A$ and a basis for $B_0$ where $x$ is one
  of the generators for $B_0$.  This corresponds to a sphere system in
  $M$ which decomposes as $\Sigma_A \cup \Sigma_{B_0}$; the sphere
  $S_0$ separates the two sets $\Sigma_A$ and $\Sigma_{B_0}$.  In
  terms of these sphere systems, we can describe a homeomorphism that
  takes $S_0$ to (a sphere isotopic to) $S_1$.

  Denote by $\Sigma_\gamma$ the ordered set of spheres (all in
  $\Sigma_A$) pierced by $\gamma$ starting from the basepoint.  Cut
  $M$ open along the sphere $\beta$ corresponding to the generator $x$
  and via a homotopy push the boundary sphere $\beta^-$ through the
  spheres in $\Sigma_\gamma$ in order, dragging $S_0$ along.  After
  regluing $\beta^+$ and $\beta^-$, the image of $S_0$ is $S_1$ and
  the sphere $\beta$ now corresponds to $xb$.  By shrinking $\beta^-$
  and $S_0$, we can assume that homotopy is the identity on $\tau_0$
  and $\gamma$.  Thus, we have a homeomorphism taking $S_0$ to $S_1$,
  $S_0 \cup \gamma$ to $S_1 \cup \gamma$ and is the identity on
  $\tau_0$.  As a homeomorphism takes a regular neighborhood to a
  regular neighborhood, $\tau_0$ is homotopic to $\tau_1$.

To see a schematic picture of this homotopy we refer the reader to the Figure \ref{fig:slide}. In the first picture, the black sphere (an example of $S_0$) and a neighborhood of the red loop gives a torus ($\tau_0$). And this torus is homotopic to the torus obtained from the last picture by taking a neighborhood of the black sphere (now an example for $S_1$) and the red loop.

A similar argument works if we replace $x$ by $xb^{-1}$, $bx$ or $b^{-1}x$.

The general case now follows as we can transform $B_0$ to $B_1$ by a finite sequence of the above transformations plus changes of basis that do not affect the associated spheres.
Indeed, the subgroup of automorphisms of $B$ that fix $b \in B$ is generated by the Nielsen automorphisms that fix $b$( \cite{ar:Wade14}, Theorem 4.1).

\begin{figure}
\centering{\includegraphics[width=0.55\textwidth]{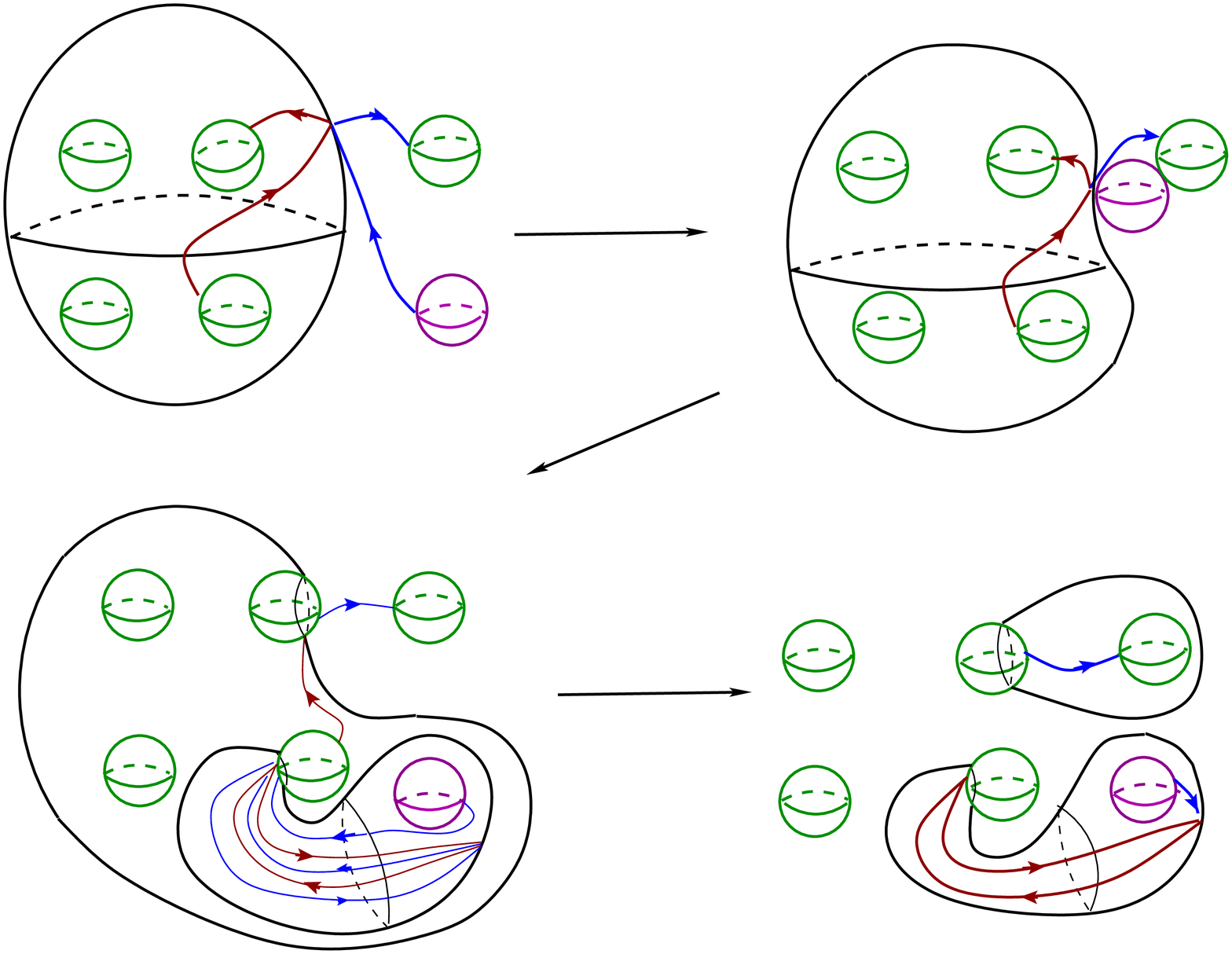}}
\caption{The homotopy which slides the pink sphere along the red loop $\gamma$ representing $b$  where $\pi_1(M)=\la a,b,c\ra$. In the first picture, a sphere (black) and the b loop(red) is given where the base point is on the sphere. c is depicted blue in the picture.}
\label{fig:slide}
\end{figure}

Finally, we need to consider the possibility that $F = A_0 * \la b \ra * B_0$ where $A = \la A_0,b \ra$ and $B = \la B_0,b \ra$.  Let $S_A$ and $S_B$ be the spheres representing the splittings $A*B_0$ and $A_0*B$ respectively, fix loops $\gamma_{A}$ and $\gamma_{B}$ representing $b$ and consider the neighborhoods  $s_{A}$  of $S_{A} \cup \gamma_{A}$ and $s_{B}$ of $S_{B} \cup \gamma_{B}$. Since each of these neighborhoods give a torus and a sphere, we have tori $\tau_{A}$ and $\tau_B$ that both represent the splitting $A*_{\la b \ra} B$.  In this case as a component of $M - (S_{A} \cup S_{B})$ is $S^1 \times S^2$ with two balls removed, it is easy to see that $\tau_A$ and $\tau_B$ are homotopic.  Indeed, let us model $S^1 \times S^2$ as the region between the spheres of radius 1 and 2 in $\RR^3$ after identifying the boundary spheres.  Remove a ball of radius $\frac{1}{4}$ at each of the points $(0,0,3/2)$ and $(0,0,-3/2)$.  For $\gamma_{A}$ we can choose the intersection with the positive $z$--axis; for $\gamma_{B}$ we can choose the intersection with the negative $z$--axis.  Then clearly the torus obtained from the intersection with the $xy$--plane is homotopic to both $\tau_A$ and $\tau_B$. For a simple example see Figure \ref{fig:slide2}.

\begin{figure}
\centering{\includegraphics[width=0.3\textwidth]{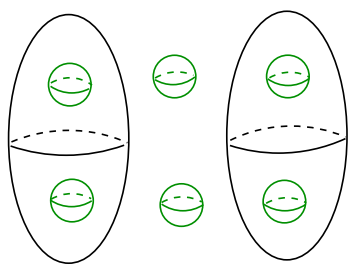}}
\caption{In this example $F_3=\la a\ra*\la b\ra * \la c \ra$ and $A=\la A_0,b\ra$ with $A_0=\la a \ra $ and $B=\la b, B_0\ra$ with $B_0=\la c\ra$. In the picture spheres $S_A$ and $S_B$ are seen in black.}
\label{fig:slide2}
\end{figure}

\medskip \noindent {\bf Case 2:} We now consider the case of an HNN-extension $F_k = A*_{\la b \ra}$.  By Swarup's Theorem~\cite{swarup}, there is a free factorization $A = A_0*\la t^{-1}bt \ra$ for some $t \in F_k$, such that $A_0$ is a co-rank 1 free factor of $F_k$ and such that $b \in A_0$.  Let $S \subset M$ be an embedded (non-separating) sphere representing the splitting $F_k = A_0*_{ \{ 1 \} }$.  We fix a basepoint $p \in M$ and assume it lies on $S$.  As $b \in A_0$, there is an embedded loop $\gamma\subset M$ that represents $b \in F$ and only intersects $S$ at $p$.  Further, both ends of $\gamma$ are on the same side of $S$ and so a neighborhood of $s = S \cup \gamma$ gives a torus.  Let $\tau$ be this neighborhood of $s$ and as in Case 1, it is clear that the splitting associated to $\t$ is the original splitting.

Given another torus $\tau' \subset M$ that represents the same splitting, we can compress $\tau'$ to a union of a sphere and a loop $s' = S' \cup \gamma'$ such that $S'$ represents a splitting of the form $A_{1} *_{\{1\}}$ where $A = A_{1} * \la t b t^{-1} \ra$.  Then as in Case 1, there is a sequence of transformations taking $A_{0}$ to $A_{1}$ that do not change the homotopy type of the corresponding torus.
\end{proof}

\section{geometric intersection and relative twisting using ends of $\widetilde{M}$}

\subsection{The intersection criterion and the relative twisting number}

The Guirardel core $\calC$ is a way of assigning a closed, connected CAT($0$) complex to a pair of splittings which counts the number of times the corresponding Bass-Serre trees intersect. Guirardel's version unifies several notions of  intersection number in the literature including the one given for two splittings of finitely generated groups given by Scott and Swarup (\cite{ScottSwarupInt}). For two $F_k$--trees $T_0$ and $T_1$ the core is roughly the main part of the diagonal action of the $F_k$ on $T_0\times T_1$. For the details we refer the reader to \cite{Gui:Core} and \cite{bestclay1}.

\begin{definition}Let $T$ be a tree and $p$ a point in it. A \textit{direction} is a connected component of $T-p$. Given two trees $T_0$ and $T_1$, a \textit{quadrant} is a product $\delta\times \delta^{'}$ of two directions $\delta \subset T_0$ and $\delta^{'}\subset T_1$.
\end{definition}

We fix a basepoint $\ast=(\ast_0, \ast_1)$ in $T_0\times T_1$ and we say that a quadrant $Q$ is \textit{heavy} if there exists a sequence $\{g_n\}$ in $F_k$ such that $g_{n}(\ast)\in Q $ for every $n$ and $d_{T_i}(\ast_i, g_n(\ast_i))\rightarrow \infty$ as $n\rightarrow \infty$. A quadrant is \textit{light} if it is not heavy.

\begin{definition}Let $T_0$ and  $T_1$ be to $F_k$--trees. The Guirardel core $\mathcal{C}$ defined  as
\[
  \mathcal{C}=\mathcal{C}(T_0\times T_1)= (T_0\times T_1)-\cup_IQ
\]
where $I$ is over all of the light quadrants.
\end{definition}

Let $p$ be a point in $T_0$. Then, $\mathcal{C}_p=\{x\in T_1\mid (p,x)\in  \mathcal{C}\}$ is a subtree of $T_1$ called the \textit{slice of the core} above the point $p$. The slice which is a subtree of $T_0$ is defined similarly.


Given two trees $T_0$ and $ T_1$,
we define the \textit{Guirardel intersection number} between them by:
\[
  i(T_0, T_1)= vol(\mathcal{C}/F_k)
\]
where the right hand side is given by the volume of the action of $F_k$ on the Guirardel core $\mathcal{C}(T_0\times T_1)$  for the product measure on $T_0\times T_1$. Note that for simplicial trees $T_0$ and  $T_1$ this volume is  the number of the $2$-cells in $\mathcal{C}/F_k$, which will be our case.



Given a maximal sphere system $\Sigma$ in $M$ and a homotopy class $\bf S$ of a sphere $S$, the intersection number between a sphere $\bf S$  and $\Sigma$ is the number $i(\bf S ,\Sigma)$ of components of intersection of a normal representative of $S$  with spheres of $\Sigma$ when the intersections are transversal (\cite{H1}). This is called the \textit{geometric} intersection number.

We need to understand how geometric intersection numbers between two sphere systems and the Guirardel intersection numbers between corresponding trees are related. Each sphere in $\widetilde{M}$ corresponds to an edge in the associated Bass--Serre tree and a geometric intersection between two spheres will give a square in product of the Bass--Serre trees. If the intersection is essential, this square is in the Guirardel core. By the definition of  Guirardel core, a square which is in the core is in a heavy quadrant. As a consequence there exists $4$ unbounded disjoint regions in $\widetilde{M}$ corresponding to such a square. On the other hand, each sphere $\widetilde S$ gives two disjoint sets $E^+(\widetilde S)$ and $E^-(\widetilde S)$ of ends of  $\widetilde{M}$. Thus whenever two spheres $\widetilde S_1$ and $\widetilde S_2$ intersect essentially in $\widetilde{M}$ there are four disjoint sets of ends $E^+(\widetilde S_1)\cap E^+(\widetilde S_2)$, $E^+(\widetilde S_1)\cap E^-(\widetilde S_2)$, $E^-(\widetilde S_1)\cap E^+(\widetilde S_2)$, $E^-(\widetilde S_1)\cap E^-(\widetilde S_2)$  of $\widetilde M$ in the complement of the intersection circle.

Finally we have the following definition which we will use in the next section.
\begin{definition}We will call the existence of four disjoint sets of ends of $\widetilde M$ in the complement of an  intersection circle between two spheres the \textit{intersection criterion} for that intersection circle.
\end{definition}

According to the discussion above, an intersection circle between two spheres is  essential if and only if we have intersection criterion satisfied  for that intersection circle.

Now we will define the \textit{relative twisting number} for two intersecting sphere systems $\Sigma_1$ and $\Sigma_2$.

Given an axis of an element in $\widetilde{M}$, there are two ends of $\widetilde{M}$ which are fixed by this axis.

\begin{definition}A sphere $\widetilde{S}$ is said to intersect an axis $\ba$ whenever the two ends of $\widetilde{M}$ determined by $\ba$ are separated by the two disjoint sets $E^+(\widetilde{S})$, $E^{-}(\widetilde{S})$ of ends corresponding to $\widetilde{S}$.
\end{definition}

Now we will define the relative twisting number between $\Sigma_1$ and $\Sigma_2$ relative to an element $a\in F_k$.
This number is meaningful when both sphere systems intersect an axis $ba$ of $a$ in $\widetilde M$. The definition of geometric relative twisting of \cite{CP2} is for two points in Outer space, which are simple sphere systems in our setting. For our purposes we will translate their definition to a one which is stated in terms of ends of $\widetilde {M}$.

We will start by the definition given in \cite{CP2}.

\begin{definition} Let $\Sigma_1$ and $\Sigma_2$ be two simple sphere systems in $M$ and  $T_{\Sigma_1}$ and $T_{\Sigma_2}$ the corresponding Bass--Serre trees. Let $e_{\widetilde S_1}$ and $e_{\widetilde S_2}$ be two edges in  $T_{\Sigma_1}$ and $T_{\Sigma_2}$ corresponding to spheres $S_1\in \Sigma_1$ and $S_2\in \Sigma_2 $, respectively.  Then, for an element $a\in F_k$, the relative twisting $tw_{a}(\Sigma_1, \Sigma_2$) of $\Sigma_1$ and $ \Sigma_2$ relative to $a$ is defined to be,
\[
  tw_a(\Sigma_1, \Sigma_2):=\max_{\mathclap{\substack{ e_{\widetilde S_1}\subset {T^a_{\Sigma_1}}\\
  e_{\widetilde S_2}\subset T^a_{\Sigma_2}}}} \,\,\{\,k \mid a^ke_{\widetilde S_1}\times e_{\widetilde S_2}\in \calC \,\, \text{and}\,\, e_{\widetilde S_1}\times e_{\widetilde S_2}\in \calC \}
\]
where $T^a_{\Sigma_1}$ and $T^a_{\Sigma_2}$ are the sets of edges of $T_{\Sigma_1}$ and $T_{\Sigma_2}$ respectively whose elements intersect a fixed axis $\ba$ of $a$.

\end{definition}

Using the fact that a geometric intersection between two spheres is a square in the Guirardel core and hence it gives a separation of ends of $\widetilde M$ into $4$ nonempty disjoint sets, we tailor this definition to the one below to suit our needs.

\begin{definition} For $i\in \{1,2\} $,  let $\widetilde S_i$ be two spheres and $E^{\mp}({\widetilde S_i})$ be the set of ends of $\widetilde M$ separated by these spheres. Assume that both spheres intersect an axis $\ba\in F_k$. Then the relative twisting $tw_{a}(\widetilde S_1, \widetilde S_2)$ of $ \widetilde S_1$ and $\widetilde S_2$ relative to $a$ is defined to be

\begin{equation*}
\begin{aligned}
tw_{a}(\widetilde S_1, \widetilde S_2):={} & \max \{k\in \mathbb{Z} \mid E^{\mp}({\widetilde S_i})\,\cap E^{\mp}({a^k\widetilde S_j})\neq\emptyset \,\,\text{whenever}\,E^{\mp}({\widetilde S_i})\,\cap\, E^{\mp}({\widetilde S_j})\neq\emptyset, \\
 &\,\{i,j\}\in \{1,2\}\}
\end{aligned}
\end{equation*}
\end{definition}


\section{Dehn twist along a torus: The Geometric Picture}

\subsection{Definition of a Dehn twist along a torus }

We will now give the definition of the Dehn twist homeomorphism about a torus in $M$, a description of the action of such a homeomorphism on spheres in $M$, and a description of the action on $F_k$.

To define a Dehn twist along an imbedded torus $\tau$, we will take a parametrized tubular  neighborhood of the torus in $M$.

\begin{figure}
\centering{\includegraphics[width=0.5\textwidth]{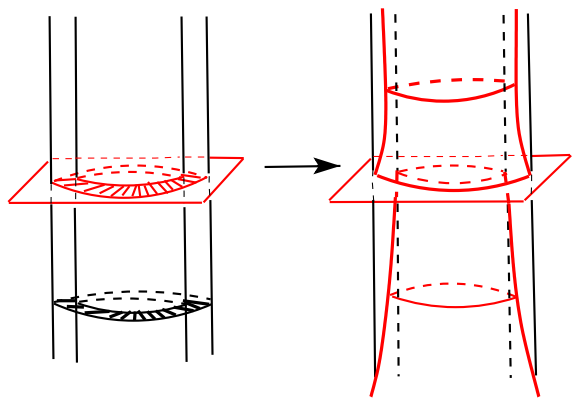}}
\caption{The image of the intersection annulus under a Dehn twist along the thick black torus.}
\label{fig:NDT1}
\end{figure}

\begin{definition}
\label{Def:dehntwist}
Let $\tau: \mathbb{R}/\mathbb{Z} \times \mathbb{R}/\mathbb{Z} \times [0,1] \rightarrow M$ be an imbedding such that $\tau(\{0\}\times \mathbb{R}/\mathbb{Z}\times \{0\})$ bounds a disk in $M$. Denote the associated torus $\tau(\mathbb{R}/\mathbb{Z}\times \mathbb{R}/\mathbb{Z}\times \{0\})$ with $\tau$. The \emph{geometric Dehn twist} $D_{\t}$ along the torus $\t$ is the homeomorphism of $M$ that is the identity on the complement of the image of the map $\tau$ and for which a point $p = \t(x,y,t)$ is sent to $\t(x+t,y,t)$.
\end{definition}
Here, the direction of the associated torus which bounds a disk in $M$ will be the \textit{meridional} direction and the other one will be called \textit{longitude} direction. For a geometric description of the twist, we refer the reader to Figure \ref{fig:NDT1}.

Now, an ambiguity might arise in the definition of a geometric Dehn twist when it comes to determining a longitude curve for the parametrization. But two such choices differ by a map $x \mapsto x+ny $ for some integer $n$ and a twist along the meridional direction. By work of Laudenbach (\cite{LAUD}) meridional direction does not give a non-trivial homeomorphism since twists along meridional direction correspond to twists along  $2$--spheres in $M$. It is known  that such mapping classes act trivially on $F_k$ hence they are in the kernel of the homomorphism $MCG(M)\rightarrow \Out(F_k)$ (\cite{McCHat}). Hence, the induced outer automorphism $D_{\t*}$ from the geometric Dehn twist $D_\t$ is independent of the parametrization of the neighborhood of the torus (image of the map $\tau$).

Now, assume that we have a loop $\xi$ intersecting a torus $\t$ transversely. Then the image $D_\t(\xi)$ of the loop under the geometric twist $D_\t$ is obtained as follows: we surger the loop at the intersection point and insert a loop $\beta^+$ or $\beta ^-$ representing a generator of $\pi_1(\t)$ in $\pi_1(M)$, depending on which side of the torus the intersection point is. Hence the induced automorphism conjugates $\xi$ with one of $\beta^+$ or $\beta^-$ and fixes the other. If $\t$ is non-separating, then the stable letter is multiplied by $\xi$.

This coincides with the action of $D_{\t*}$ on the corresponding splitting: When the splitting is amalgamated product of the form $A*_{\la c \ra} B$, $A$ is fixed whereas $B$ is conjugated by the generator of the fundamental group of torus $\tau$ in $M$. (Roles of $A$, $B$ might be changed.) When we have $A*_{\la tct^{-1} = c' \ra}$, $D_{\t*}$ fixes $A$ and $t$ is multiplied by $c$.

As a summary  we have,
\begin{lemma}\label{geoDT} Let $\t$ be  an imbedded torus and $D_\t$ the associated geometric Dehn twist. Then, $D_{\t *}= D_Z$ where $D_{\t *}$ is the Dehn twist automorphism induced by the homeomorphism $D_\t$ and $D_Z$ is the Dehn twist automorphism given by the $\ZZ$--splitting $Z$ associated to the torus $\t$.
\end{lemma}


From the lemmata \ref{geoDT} and \ref{tree}, we easily deduce the following.

\begin{proposition} Let $\t_1$ and $\t_2$ be two homotopic tori.
Then up to conjugacy  we have $D_{\t_1*} = D_{\t_2*}$  where $D_{\t_1*}$ and $D_{\t_2*}$ are the Dehn twist automorphisms induced by the geometric Dehn twists $D_{\t_1}$ and $D_{\t_2}$.
\end{proposition}
\begin{proof}Since homotopic tori give equivalent splittings by Lemma \ref{tree}, the Dehn twist automorphisms are equal by the definition. Then by lemma \ref{geoDT} the corresponding Dehn twists induced by the geometric Dehn twists are equal.
\end{proof}




For our purposes, we need to find a lower bound on $tw_a(G, D^n_{T}(G))$ for a given simple sphere system $G$ and a $\mathbb{Z}$--splitting $T$. Now we have $D_T=D_{\t}$ where $\t$ is the essential imbedded torus given  by the $\mathbb{Z}$--splitting $T$ for which the Dehn twist is defined and  $a$ is the generator of the image of the fundamental group of $\tau$ in $\pi_1(M)$ under the map induced from the imbedding $\iota: \t \rightarrow M$. Now we take a maximal sphere system containing $G$ and homotope $\t$ to be normal with respect to this sphere system. Then $\t$ intersects the spheres of $\Sigma$ minimally, by \cite{F1}. Now we take a lift of the torus which has an axis a conjugate of $a$. The relative twisting number counts the number of iterates of a sphere which intersect image of another sphere under a Dehn twist along an axis. Hence we have,  

\[
 tw_a(G, D^n_{T}(G))=tw_{\t}(G, D^n_{\t}(G)).
\]

\section{Lower Bound on the Relative Twisting Number}
In this section we will prove that the relative twisting number of a simple sphere system and its Dehn twisted image has a lower bound, which is linear with respect to the power of the Dehn twist. We first have the following introductory lemma, recall that for a sphere system in $M$, \textit{simple} means that the complementary components in $M$ are simply connected. 

Given a sphere $S$ and a torus $\t$ an intersection circle of $S\cap \t$ which does not bound a disk in $\t$ is called a \textit{meridian} in $\tau$.
\begin{lemma}\label{pret}Given an imbedded essential torus  $\tau$ and  a simple sphere system $G$, there exists a sphere $\bS\in G$ such that at least one intersection between $\bS$ and $\tau$ is meridian in $\tau$.
\end{lemma}
\begin{proof} We first complete $G$ to a maximal sphere system and homotope $\t$ to be normal with respect to this maximal sphere system. Then $\t$ intersects minimally every sphere of $G$ it intersects (\cite{F1}). Assume that no isotopy class of spheres in $G$ intersect $\t$ in a way that the intersection is meridian in $\t$. On the other hand, since $\t$ is essential, the image of the fundamental group of $\tau$ in $\pi_1(M)$ is non-trivial. Hence, the core curve of $\t$ exists and by the assumption it must be contained in a complementary component of $G$ in $M$. But this is not possible since $G$ is simple.
\end{proof}

\begin{theorem}\label{twistlemma} Let $T$ be a $\mathbb{Z}$--splitting of the free group $F_k$, $\tau$ the associated torus and $D_{\tau}$ the Dehn twist along ${\tau}$. For $G \in CV_k$ and  $n\geq 2$,
\[
tw_{\tau} \big{(}G , D^n_{\tau}(G) \big{)} \geq n-1.
\]
\end{theorem}

\begin{figure}
\setlength{\unitlength}{0.01\linewidth}
\begin{picture}(100,45)
\put(10,0){\includegraphics[width=0.75\textwidth]{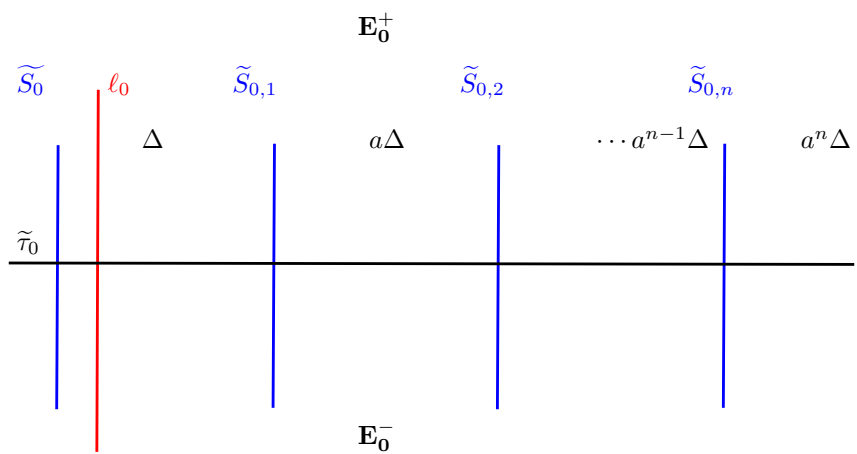}}

\put(41,37){$\mathbf{E^+_0}$}
\put(41,1){$\mathbf{E^-_0}$}
\put(11,32){$\color{blue}{\widetilde{S_0}}$}
\put(19,32){$\color{red}{{\ell_0}}$}
\put(11,18){$\widetilde\tau_0$}
\put(30,32){$\color{blue}{\widetilde S_{0,1}}$}
\put(50,32){$\color{blue}{\widetilde S_{0,2}}$}
\put(70,32){$\color{blue}{\widetilde S_{0,n}}$}
\put(22,27){$\Delta$}
\put(42,27){$a\Delta$}
\put(62,27){$\cdots a^{n-1}\Delta$}
\put(80,27){$a^n\Delta$}
\end{picture}
\caption{Fundamental domains, sets of ends separated by $\ttau_0$ and an arc $\ell_0$ connecting them. }
\label{fig:endtwist1}
\end{figure}

\begin{figure}
\setlength{\unitlength}{0.01\linewidth}
\begin{picture}(100,35)
\put(10,0){\includegraphics[width=0.75\textwidth]{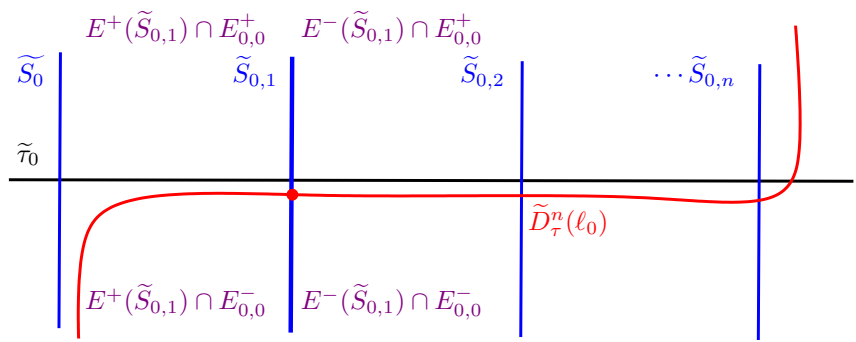}}
\put(56,10){$\color{red}{\widetilde D^n_{\tau}(\ell_0)}$}
\put(17,27){$\color{violet}E^+(\widetilde S_{0,1})\cap E_{0,0}^+$}
\put(36,27){$\color{violet}E^-(\widetilde S_{0,1})\cap E_{0,0}^+$}
\put(17,3){$\color{violet}E^+(\widetilde S_{0,1})\cap E_{0,0}^-$}
\put(36,3){$\color{violet}E^-(\widetilde S_{0,1})\cap E_{0,0}^-$}
\put(11,16){$\widetilde\tau_0$}
\put(11,23){$\color{blue}{\widetilde{S_0}}$}

\put(30,23){$\color{blue}{\widetilde S_{0,1}}$}
\put(50,23){$\color{blue}{\widetilde S_{0,2}}$}
\put(67,23){$\color{blue}{\cdots\widetilde S_{0,n}}$}

\end{picture}
\caption{The image of $\ell_0$ under Dehn twisting and sets of ends corresponding to the intersection $\widetilde D^n_{\tau}(\widetilde S_0)\cap \widetilde S_{0,1}$.}
\label{fig:endtwist3}
\end{figure}

\begin{proof} We will prove this theorem by taking $G$ as a simple sphere system instead of a marked metric graph. Let $\Sigma$ be a maximal sphere system completing $G$ and homotope $\tau$ to be the normal with respect to $\Sigma$.

Now, let $S\in G$ be such that $S$ and $\t$ intersect in a way that $\mu=S\cap \t$ is meridian in $\tau$. The existence of a meridional intersection circle is given by Lemma \ref{pret}. As before, let $D_{\tau}$ be the Dehn twist along the normal torus $\tau.$ More precisely, let $N(\tau)$ be a tubular neighborhood of $\tau$ in which $D_{\tau}$ is supported. Let $\ttau$ be the full preimage of $\tau$ in $\widetilde{M}$ and $\ttau_0$ a component, $N(\ttau)$ be the full preimage of $N(\tau)$ and  $N(\ttau_0)$ be the component containing $\ttau_0$. Let $\widetilde S$ be the full preimage of $S$, and $\widetilde{S_0}$ a component such that $\widetilde{S_0}\cap \ttau_0= \widetilde\mu_0$ , a lift of $\mu$. Let us call $a$ the generator of $\pi_1(\tau)$ corresponding to $\ttau_0$, hence we have a covering transformation $a: \widetilde{M} \rightarrow \widetilde{M}$ which stabilizes $\ttau_0$. Let $\Delta$ be the region between $\widetilde S_0$ and $a\widetilde S_0$ which is the fundamental domain of $\la a\ra$ on $\widetilde {M}$. Set $\widetilde S_{0,j}=a^j \widetilde S_0$. Then $a^j\Delta$ is the region bounded by $\widetilde S_{0,j}$ and $\widetilde S_{0, j+1}$.
\vspace{0.3cm}

Since $tw_{\tau} \big{(}G , D^n_{\tau}(G) \big{)} \geq tw_{\t} \big{(}S , D^n_{\t}(S) \big{)}$, it is sufficient to prove that

\[
tw_{\t} \big{(}S , D^n_{\t}(S) \big{)} \geq n-1.
\]

Let us denote by $E(\widetilde{M})$ the ends of $\widetilde{M}$. As it was discussed in Section 2.2, there is a pair of ends of $\widetilde{M}$ fixed by $a$ and $\ttau_0$ separates the remaining set of ends into two disjoint sets $E_0^{+}$ and $E_0^{-}$. Since $\ttau_0$ is separating in $\widetilde M$, there is a ray $\ell$ which connects an end $\sf{e^{+}}$$\in E_0^+$ to an end $\sf{e^{-}}$$\in E_0^-$, intersecting $\ttau_0$ only once. Observe that since $\t$ is essential, there is always such ray which is disjoint from $\ttau- \ttau_0$.

Let $X^+$ and $X^{-}$ be the components of $\widetilde{M}-N(\ttau)$ whose closures meet  $N(\ttau_0)$. Since $D_{\tau}$ is identity on ${M}-N(\tau)$ we choose a lift $\widetilde D_{\tau}$ which is the identity on $X^{-}$ and  a translation on $X^{+}$. Without loss of generality we may assume this is translation by $a$. Hence, $\widetilde D_{\tau}(\sf{e^{-}})= \sf{e^{-}}$ and $\widetilde D^{m}_{\tau}(\sf{e^{+}})= a^{m}\sf {e^{+}}$.

Let $E_{0,0}^- \subset E_0^- $ and $E_{0,0}^+ \subset E_0^+ $ be two disjoint sets of ends in $\Delta$. Now, as described above, in this fundamental domain there is a line $\ell_0$ which connects an end $\sf{e_0^{-}}$ in $E_{0,0}^-$ to an end $\sf{e_0^{+}}$ in $E_{0,0}^+$ intersecting $\ttau$ once in $\ttau_0$. (See the Figure \ref{fig:endtwist1}). Now, after $n$  times Dehn twisting along $\ttau$, the image ray $\widetilde D^{n}_{\tau}(\ell_0)$ will connect the point $\sf{e_0^{-}}$ in $E_{0,0}^-$ to the point $\widetilde D^{n}_{\tau}(\sf{e_0^{+}})=a^n \sf e_0^+={e_{n}^{+}}$ in $E_{0,{n-1}}^+= \widetilde D^n_{\tau}(E_{0,0}^+)$. (See the schematic picture Figure \ref{fig:endtwist3}).

On the other hand, let us denote by $E^+(\widetilde S_{0,s})$ and $E^-(\widetilde S_{0,s})$ the two disjoint sets of ends corresponding to the sphere $\widetilde S_{0,s}$,  for $s \in \{1,\cdots ,n \}$. Now, without loss of generality assume that  $\sf{e_0^{-}}$$\in E^-(\widetilde S_{0,1})$. Then, since the image ray $\widetilde D^{n}_{\tau}(\ell_0)$ still intersects $\ttau_0$ only once, and does not intersect neither $\widetilde S_{0}$ (this would create a bigon) nor any other lift of the torus, $\widetilde D^{n}_{\tau}(\sf{e_0^{+}})={e_{n}^{+}}$$\in E^+(\widetilde S_{0,1})$. So for $s=1$ we have $4$ disjoint, non empty sets of ends $E^+(\widetilde S_{0,1})\cap E_{0,0}^- $, $E^+(\widetilde S_{0,1})\cap E_{0,0}^+$, $E^-(\widetilde S_{0,1})\cap E_{0,0}^-$ and $E^-(\widetilde S_{0,1})\cap E_{0,0}^+ $. By intersection criterion,  this shows that $\widetilde D^n_{\tau}(\ell_0)$ intersects $\widetilde S_{0,1}$. Since $E^+(\widetilde S_{0,1})\subset E^+(\widetilde S_{0,n})$ and $E^-(\widetilde S_{0,n})\subset E^-(\widetilde S_{0,0})$ we similarly have the nonempty disjoint sets of ends $E^+(\widetilde S_{0,n})\cap E_{0,{n-1}}^- $, $E^+(\widetilde S_{0,n})\cap E_{0,{n-1}}^+$, $E^-(\widetilde S_{0,n})\cap E_{0,{n-1}}^-$ and $E^-(\widetilde S_{0,n})\cap E_{0,{n-1}}^+ $. Again by intersection criterion this means that $\widetilde D^n_{\tau}(\ell_0)$ intersects $\widetilde S_{0,n}$ as well. As a result, $\widetilde D^n_{\tau}(\ell_0)$ intersects all iterates $\widetilde S_{0,s},\,\, s \in \{1,\cdots ,n \}$ of $\widetilde S_0$.

Hence $\widetilde D^n_{\tau}(\widetilde S_0)$ intersects all $n-1$ iterates of $\widetilde{S_0}$. More precisely,
\[ \widetilde D^{n}_{\tau}(\widetilde S_0)\cap {\widetilde S_{0,j}}\neq\emptyset \,\,\text{for}\, j=1,\cdots, n,\,\,\, n\geq2,
\]

and thus we conclude that,
\[
tw_{\t} \big{(}S , D^n_{\t}(S) \big{)} \geq n-1.
\]

\end{proof}

\section{The Almost Fixed Set}

In this section we will prove the following theorem  which says that there is an upper bound on the diameter of the almost fixed set of a Dehn twist and this upper bound depends only on the rank of the free group.
\begin{diamtheorem}Let $T$ be a $\mathbb{Z}$--splitting of $F_k$ with $k >2$,  and $D_{T}$ denote the corresponding Dehn twist. Then, for all sufficiently large constants $C$, there exists a $C'=C'(C,k)$ such that the diameter of the almost fixed set
\[
F_C= \{x \in FF_k: \exists n \neq 0 \,\text{such that}\,\, d(x, D_{T}^n(x)) \leq C \} \]
corresponding to $ \langle D_{T} \rangle$ is bounded above by $C'$.
\end{diamtheorem}

\subsection{Finding the folding line}
For $G_1$ and $G_2$ in $CV_k$, let $m(G_1, G_2)$ be the infimum of the set of maximal slopes of all \textit{change of markings} (a map linear on edges) $f: G_1\longrightarrow G_2$. Then we define a function
$d_L: CV_k\times CV_k \longrightarrow \mathbb{R}_{\geq 0}$ by
\[d_L(G_1, G_2)= \log m(G_1, G_2).
\]
Despite being non-symmetric, since this is its only failure to be distance,  $d_L$ is referred as the \textit{Lipschitz metric} on $CV_k$.

For an interval $I\subset \mathbb{R}$ the folding lines $\widetilde{g}: I \longrightarrow CV_k$ are the paths connecting any given two points in the interior of $CV_k$, obtained as follows:

For $G_1, G_2 \in CV_k$ let $f: G_1\longrightarrow G_2$ be a change of marking map whose Lipschitz constant realizes the maximal slope. We find a path based at $G_1$ which is contained in an open simplex of unprojectivized outer space and  parametrize it by arclength. Then we concatenate this path with another geodesic path outside the open simplex obtained by folding process. The resulting path  $g: [0, d_L(G_1, G_2)]\longrightarrow CV_k$ is a geodesic by Francaviglia and Martino (\cite{FM1}), which we will call \textit{folding line}.

For a given $F_k$--tree $\Gamma$ and an element $a\in F_k$, let $\ell_{\Gamma}(a)$ denote the minimal translation length of $a$ in $\Gamma$.
To prove the Theorem \ref{diam}, we use the following result from \cite{CP2}:
\begin{theorem}[\cite{CP2}]\label{clayp2} Suppose $G,G^{\prime} \in CV_k$ with $d = d_L(G,G^{\prime})$ such that
$tw_a(G,G^{\prime}) \geq n+2$ for some $a \in F_k$. Then there is a geodesic (folding line) $g: [0, d]\rightarrow
CV_k$ such that $g(0) = G$, $g(d) = G^{\prime}$ and for some $t \in [0, d]$, we
have $\ell_{g(t)}(a) \leq \frac{1}{n}$. In other words, $g([0, d])\cap CV_k^ {1/n}
\neq \emptyset$
\end{theorem}

\subsection{Converting short length to distance}

Let $\pi$ be the coarse projection $\pi:CV_k\rightarrow FF_k$.
\begin{lemma}[Lemma 3.3 \cite{BestHyp}]\label{L1} Let $a \in F_k$ be a simple class, $G$ a point in $CV_k$ so that the loop corresponding to $a$ in $G$ intersects some edge $\leq m$ times. Then,

\[ d_{FF_k}({\alpha},\pi(G))\leq 6m+13
\]
where ${\alpha}$ is the smallest free factor containing the conjugacy class of $a$.
\end{lemma}

Using this lemma we prove the following:

\begin{lemma}\label{L2} Let $\alpha$ be a free factor containing the conjugacy class of an element $a \in F_k$, $G$ a point in $CV_k$ and $\ell_{G}(a)\leq m$. Then, there is a  constant $B$ and a number $A$ depending only on the rank of the free group such that

\[
d_{FF_k}({\alpha},\pi(G))\leq Am+B \]

\end{lemma}

\begin{proof} Let $e$ be the edge of $G$ with the longest length.
Hence, $\ell(e)\geq 1/(3k+3)$. Then, $\alpha$ crosses $e$ less than $(3k+3)m$ times. Therefore, by Lemma \ref{L1}, $d_{FF_k}({\alpha},\pi(G))\leq 6(3k+3)m+13$. Here, $A= 6(3k+3)$ clearly depends only on the rank of the free group.
\end{proof}

\subsection{Proof of Theorem \ref{diam}}

\begin{proof}[Proof of Theorem 1.2] Given $x \in F_C$, let us assume that $G\in CV_k$ is a point which is projected to $x$. We will write $\pi(G)=x$. Let $\tau$ be the essential imbedded torus in $M$ corresponding to the given $\ZZ$--splitting $T$.

Let $n$ be such that $d_{FF_k}(G, D_T^n(G))\leq C$. Up to replacing $C$ by a constant we can assume that $n\geq 4$.
By Theorem \ref{twistlemma}, we have $tw_{\tau} \big{(}G , D^n_{\tau}(G) \big{)} \geq n-1$. Since $tw_{\tau} \big{(}G , D^n_{\tau}(G) \big{)}=tw_{a} \big{(}G , D^n_{T}(G) \big{)}$ where $a\in F_k$ represents the core curve of the torus $\tau$, we can use Clay-Pettet Theorem \ref{clayp2}. Thus there is a folding path, ${G_t: [0,d]\rightarrow CV_k}$ such that $G_0=G$, $G_{d}=D_{\t}^{n}(G)$ and $a$ gets short along the geodesic $\{G_t\}_t$; for some $t \in [0,d]$,  $\ell_{G_t}(a)\leq 1/(n-3)$.

Now, by \cite{BestHyp} and \cite{KaRa}, the projection $\pi(\{G_{t}\})$ of the folding path  $\{G_t\}_{t}$ onto $FF_k$ is a quasi-geodesic in $FF_k$ between $\pi(D^n_{\tau}(G))$ and $\pi(G)=x$. Let $\alpha$ be the smallest free factor containing $a$.

Then, since $n \geq 4$ we use Lemma \ref{L2} to deduce that
\begin{center}
$d_{FF_k}({\alpha}, \pi(G_t))\leq \frac{A}{n-3}+13$
\end{center}
where $A$ is  as it is given in the same lemma. Since $\pi(\{G_{t}\})$ is uniformly Hausdorff-close to a geodesic and $d_{FF_k}(\pi(G), \pi(D^n_{\tau}(G)))\leq C$, by triangle inequality we have
\begin{center}
$d_{FF_k}(\pi(G), {\alpha})\leq \frac{A}{n-3}+C+H+13\leq A+C+H+13$
\end{center}
where $H=H(k)$ is the distance between the geodesic and unparametrized quasi geodesic $\pi(\{G_{t}\})$.
Hence we have 
\[C'=2(A+C+H+13)\].

\end{proof}

\section{Constructing Fully Irreducibles}

In this section we will prove the main theorem of this paper.
\begin{twisttheorem} Let $T_1$ and $T_2$ be two $\mathbb{Z}$--splittings of the free group $F_k$ with rank $k>2$ and $\alpha_1$ and $\a_2$ be two corresponding free factors in the free factor complex $FF_k$ of the free group $F_k$. Let $D_{1}$ be a Dehn twist fixing $\alpha_1$ and $D_{2}$ a Dehn twist fixing $\alpha_2$, corresponding to $T_1$ and $T_2$, respectively. Then there exists a constant $N=N(k)$ such that whenever $d_{FF_k}(\alpha_1, \a_2) \geq N$,
\begin{enumerate}
\item $\langle D_{1}, D_{2}\rangle \simeq F_2$.
\item All elements of $\langle D_{1}, D_{2}\rangle$ which are not conjugate to the powers of $D_{1}, D_{2}$ in $\langle D_{1}, D_{2}\rangle$  are fully irreducible.
\end{enumerate}
\end{twisttheorem}

We will start with some basic definitions and lemmata, which are standard for $\delta$--hyperbolic spaces.
\newline

Let $(X,d)$ be a metric space. For $x,y,z \in X$, the \textit{Gromov product} $(y,z)_x$ is defined as
\[ (y,z)_x:=\frac12 (d(y,x)+d(z,x)-d(y,z)).
\]

If $(X,d)$ is a $\delta$-hyperbolic space, the initial segments of length of $(y,z)_x$ of any two geodesics $[x,y]$ and $[x,z]$ stay close to each other. Namely, they are in $2\delta$-neighborhoods of each other. Hence, the Gromov product measures for how long two geodesics stay close together. This will be the characterization of the $\delta$- hyperbolicity we will use in our work as definition of being $\delta$-- hyperbolic.

Also, in this case, the Gromov product  $(y,z)_x$ approximates the distance between $x$ and the geodesic $[y,z]$ within $2\delta$:
\[
(y,z)_x \leq d(x, [y,z])\leq (y,z)_x+ 2\delta
\]
\begin{definition}A path $\sigma: I \rightarrow X$ is called a $(\lambda, \epsilon)$--quasi geodesic if $\sigma$ is parametrized by arc-length and for any $s_1$, $s_2 \in I$ we have
\[
  |s_1-s_2|\leq \lambda d(\sigma(s_1), \sigma(s_2))+\epsilon
\]
If the restriction of $\sigma$ to any subsegment $[a,b] \subset I$ of length at most $L$ is a $(\lambda, \calL)$ quasigeodesic, then we call $\sigma$ an $\ell$--local  $(\lambda, \calL)$--quasigeodesic.
\end{definition}

Let $X$ be a geodesic metric space and $Y\subset X$. We say that $Y$ is $c$--quasiconvex if for all $y_1, y_2 \in Y$ the geodesic segment $[y_1, y_2]$ lies in the $c$-- neighborhood of $Y$.

For any $x\in X$ we call $p_x\in Y$ an \textit{$\epsilon$-quasi projection } of $x$ onto $Y$ if
\[
  d(x,p_x)\leq d(x,Y)+ \epsilon.
\]
In hyperbolic spaces, quasi projections onto quasiconvex sets are \textit{quasiunique}:

\begin{lemma}\label{quasiPro} Let $X$ be $\delta$--hyperbolic metric space and $Y \subset X$ $c$--quasiconvex. Let $x\in X$, and $p_x$ and $p_{x{'}}$ be two $\epsilon$--quasi projections of $x$ onto $Y$. Then,
\[
  d(p_x, p_{x'})\leq 2c+4\delta+ 2\epsilon. \]

\end{lemma}

For a $\delta$ hyperbolic geodesic $G$-space $X$, consider the almost fixed set $X_C(g)$ corresponding to a subgroup $\la g \ra $ for $g\in G$. Then, the \textit{quasi convex hull}  $\calX _C(g)$ of $X_C(g)$ is defined to be the union of all geodesics connecting any two points of $X_C(g)$. From now on we will work with quasi convex hulls of almost fixed sets.
The following is a standard for $\delta$- hyperbolic spaces.

\begin{lemma}[\cite{KWeid1} Lemma 3.9] $\calX_C(g)$ is  $g$-invariant  and $4\delta$-quasiconvex.
\end{lemma}

The following lemma appears as Lemma 3.12 in \cite{KWeid1}:
\begin{lemma}\label{L3}
Let $(X,d)$ be a $\delta$--hyperbolic space and let $[x_p, x_q]$ be a geodesic segment in $X$. Let $p,q\in X$ be such that $x_p$ is a projection of $p$ on $[x_p, x_q]$ and that $x_q$ is a projection of $q$ on $[x_p, x_q]$. Then if $d(x_p,x_q)>100\delta$ , the path $[p,x_p]\cup [x_p, x_q]\cup [x_q, q]$ is a $(1, 30 \delta)$--quasigeodesic.
\end{lemma}

In the proof of Theorem \ref{gen} we use the Lemma \ref{L3} which assumes that the projections onto almost fixed sets exist. However, given a geodesic metric space $X$ and $Y\subset X$
when $Y$ is not closed in $X$, the closest point projection onto $Y$ may not exist. To fix this we will use quasi-projections which exist quasi-uniquely when there is quasi convexity, by Lemma \ref{quasiPro}.

Now we will state and prove the following theorem which essentially proves the Theorem \ref{twist}.

\begin{gentheorem}Let $G$ be a group acting  on a $\delta$-hyperbolic metric space $X$ by isometries and $\phi_1, \phi_2 \in G$. Suppose $C>100 \delta$ and the almost fixed sets  $\calX_C(\phi_1)$ and $\calX_C(\phi_2)$ of $\langle \phi_1 \rangle$ and $ \langle \phi_2 \rangle$ respectively have diameters bounded above by a constant $C{'}$. Then there exists a constant $C_1$ such that, whenever  $d_X(\calX_C(\phi_1),\calX_C(\phi_2)) \geq C_1$,
\begin{enumerate}
\item $\langle \phi_1, \phi_2 \rangle \simeq F_2$ and,

\item every element of $\langle {\phi_1}, {\phi_2}\rangle $ which is not conjugate to the powers of $\phi_1, \phi_2$ in $\langle {\phi_1}, {\phi_2}\rangle $ acts loxodromically in $X$.
\end{enumerate}
\end{gentheorem}

\begin{proof}

Let $p_{1}\in \calX_C(\phi_1)$ and $p_{2} \in \calX_C(\phi_2)$ be two points such that $d_X(p_1,p_2)=d_X(\calX_C(\phi_1),\calX_C(\phi_2))$.
To prove the theorem we will pick a random word $\omega$  and construct a ping pong argument involving the sets $\calX_C(\phi_1)$ and $\calX_C(\phi_2)$. The goal is to show that the iterates of $ [p_{1}, p_{2}]$ under $\omega$ give a local quasigeodesic, hence a quasigeodesic.
Without loss of generality, for $g \in\langle\phi_1 \rangle $, we will start with proving that the path  $[p_{2}, p_{1}] \cup [p_{1}, gp_{1}] \cup g[p_{1}, p_{2}]$ is a quasigeodesic.

 Let $\pi(p_{2})$ and $\pi(gp_{2})$ be quasi-projections of the points $p_{2}$ and $gp_{2}$ on the geodesic segment $[p_{1}, gp_{1}]$. Then, $p_1$ and $\pi(p_2)$ are both $4\delta$--quasi projections. This is true for $\pi(gp_{2})$ and $gp_1$ also. Since the difference is negligible we will assume $p_1$ and $gp_1$ to be closest point projections.

Now we prove that $d(p_{1},gp_{1})\geq C$. To see this suppose that $d_X(p_{1},gp_{1})<C$. We take a point $x$ on the geodesic segment $[p_1,p_2]$ such that $d(x,p_1)< \epsilon$ where $\epsilon=\displaystyle\frac {C- d(p_{1},gp_{1})}{2}$. Then,
\[
d(x, gx)\leq 2\epsilon+ d(p_{1},gp_{1})=C.
\]
which is a contradiction since $p_1$ is the closest point of $\calX_C(\phi_1)$ to $\calX_C(\phi_2)$. Since the claim is proved we apply Lemma \ref{L3} to conclude that the path $[p_{2}, p_{1}] \cup [p_{1}, gp_{1}] \cup g[p_{1}, p_{2}]$ is a quasigeodesic.

\begin{figure}
\setlength{\unitlength}{0.01\linewidth}
\begin{picture}(100,45)
\put(20,0){\includegraphics[width=0.75\textwidth]{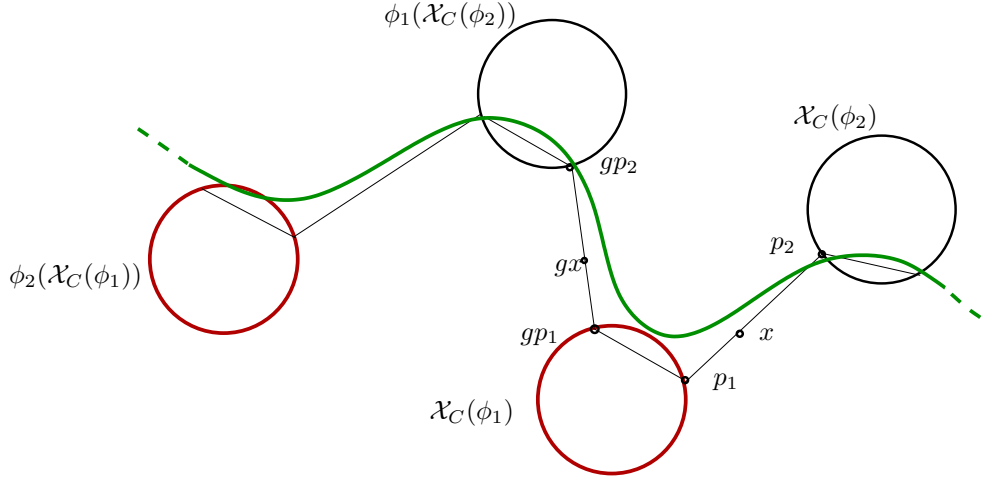}}
\put(54,12){$gp_{1}$}
\put(61,27){$gp_{2}$}
\put(76,20){$p_{2}$}
\put(71,8){$p_{1}$}
\put(75,12){$x$}
\put(57,18){$gx$}
\put(9,17){$\phi_2 (\calX_C({\phi_1}))$}
\put(78,31){$\calX_C({\phi_2})$}
\put(46,5){$\calX_C({\phi_1})$}
\put(42,40){$\phi_1(\calX_C({\phi_2}))$}

\end{picture}
\caption{The ping pong sets and a quasigeodesic between iterated points.}
\label{fig:pp1}
\end{figure}

Now let $p_{1}\in \calX_C(\phi_1)$, $p_{2}\in \calX_C(\phi_2)$ be two points as above and we take a geodesic between them. Such a geodesic, which is called a \textit{bridge} is not unique, however it is \textit{almost unique}, when two sets are sufficiently far apart ([Lemma 5.2, \cite{KWeid1}]). Hence assume that,
\[
d(p_{1}, p_{2})\geq C
.\] 
Since $C>100\delta$ this is sufficient to have a quasiunique bridge. 
Now we take a word $\omega=\phi_2^{m_l}\phi_1^{s_l} \, \cdots \, \phi_2^{m_1}\phi_1^{s_1}$ and consider the iterates of the quasiunique bridge $[p_1, p_2]$ under $\omega$. 

It is known that by the hyperbolicity of $X$, given $(\lambda, \calL)$, there exists $\ell>0$ such that a $\ell$--local $(\lambda,\calL)$--quasigeodesic is a $(\lambda', \calL')$--quasigeodesic where $\lambda'=\lambda'(\lambda, \calL, \ell)$, $L'=L'(\lambda, L, \ell)$. Since we have $(1,30\delta)$--quasigeodesic pieces, we have such $\ell=\ell(30 \delta)$. Hence we let 
 \[ d(p_{1}, p_{2})\geq C_1:=\max\{100 \delta, \ell\}.
 \]
 
 Thus the path $\gamma:= [p_{1}, p_{2}]\cup [p_1, \phi_1^{s_1}p_1]\cup [\phi_1^{s_1}p_{1}, \phi_1^{s_1}p_{2}]\cup [\phi_1^{s_1}p_{2}, \phi_2^{m_1}\phi_1^{s_1}p_{2}]\cup \cdots \cup [\omega{p_1}, \omega{p_2}]$ is a $\ell$--local  $(1,100\delta)$--quasigeodesic, which is a quasigeodesic.

In particular, for any word $\omega$ in $\langle \phi_1, \phi_2 \rangle$ we have $d(\omega(x),x)\geq |\omega|$ where $|\omega|$ denotes the syllable length, up to conjugation. Now, it follows that $\langle \phi_1, \phi_2 \rangle$ is free. Since the path which is obtained by iterating any segment between two almost fixed sets under $\omega$ is a quasigeodesic, $\omega$ is loxodromic.

\end{proof}
\begin{proof}[Proof of Theorem 1.1]
Consider the action of $\Out(F_k)$ on the free factor complex $FF_k$, which is known to be hyperbolic. Then for a sufficiently large constant $C=C(k)$, there exists a $C'=C'(C)$ such that the diameter  of the almost fixed set of a Dehn twist is bounded above by  $C'$, by Theorem \ref{diam}.
Let $C_1$ be the constant from the Theorem \ref{gen}.

Now, assume that $D_{1}$ and $D_{2}$ are the Dehn twists so that $d_{FF_k}(\alpha_1, \alpha_2)\geq 2C{'}+C_1$ where $\alpha_1$ and $\alpha_2$ are the projections of the given $\ZZ$-splittings $T_1$ and $T_2$ to $FF_k$, respectively. Then, since from Theorem \ref{diam} we have $\diam\{F_C\}\leq A+C+H+13 =C'$,
\[ d(F_C(D_1), F_C(D_2)) \geq C_1. \]
Hence, Theorem \ref{gen} applies to $\la D_{1},D_{2}\ra$ with $N= 2 \diam\{F_C\} +C_1 = 2C'+C_1$. It is clear that $N=N(k)$. As a conclusion, since loxodromically acting elements in the free factor complex are fully irreducible, every element from the group $\la D_{1},D_{2}\ra$ is either conjugate to powers of the twists or fully irreducible.
\end{proof}

\section{Constructing atoroidal fully irreducibles}
In this part we prove the following theorem which produces \textit{atoroidal} fully irreducibles. Recall  $FF_k^{(1)}$ is the 1-skeleton of the free factor complex (free factor graph) and that $\sigma: FF_k^{(1)} \rightarrow \calP_k$ is a coarse surjective map and that both graphs have the same vertex sets.

\begin{atotheorem}Let $T_1$ and $T_2$ be two $\mathbb{Z}$--splittings of $F_k$ with $k>2$ with corresponding free factors $\alpha_1$ and $\alpha_2$ and let  $D_{1}$ and $D_{2}$ be two Dehn twists corresponding to $T_1$ and $T_2$, respectively.  Then there exists a constant $N_2=N_2(k)$ such that whenever $d_{\calP_k}(\sigma(\alpha_1), \sigma(\alpha_2)) \geq N_2$ , $\langle D_{_1}, D_{_2}\rangle \simeq F_2$ and all elements from this group which are not conjugate to the powers of $D_{1}, D_{2}$ in $\langle D_{1}, D_{2}\rangle$  are atoroidal fully irreducible.
\end{atotheorem}
\begin{proof} To see that $\langle D_{1}, D_{2}\rangle \simeq F_2$ we use the Theorem \ref{gen}. To do that, we need to show that given a $\mathbb{Z}$--splitting $T$ and a constant $C$ there is another constant depending only on $C$ which bounds from above the diameter of the almost fixed set $\calP_C$ of a Dehn twist corresponding to $T$.

Let $x$ be a point in $\calP_C$ and $\sigma\pi(G)=x$ for some $G\in CV_k$. Then, as before, we use the Theorem \ref{clayp2} of Clay and Pettet to obtain a folding line $\{G_t\}_t$ between $G$ and $D^n_{\tau}(G)$ along which $a$ is short, where $a$ is the generator of the fundamental group of the torus $\tau$ in $M$. Then, since distances in the intersection graph are shorter, the Bestvina Feighn lemma \ref{L1} is valid,
\[ d_{\calP_k}(\sigma\alpha, \sigma\pi(G))\leq d_{FF_k}({\alpha},\pi(G))\leq 6m+13.
\]
Hence, we can convert the short length to distance in the intersection graph. Thus there exists constants A and B such that

\[ d_{\calP_k}({\alpha},\pi(G_t))\leq Am+B
\]

where $G_t$ is the point along which $a$ is short. Then the rest of the proof follows the same since
the image of the folding line in  $\calP_k$ is  a quasigeodesic and it is Hausdorff close to a geodesic by \cite{KaRa} and \cite{MannTh}. Hence diameter of $\calP_C$ is uniformly bounded above by a constant, and  Theorem \ref{gen} applies to $\la D_{1},D_{2}\ra$ with $N_2= 2 \diam\{\calP_C\} +C_1$.

Since in $\calP_k$ loxodromically acting automorphisms are atoroidal fully irreducible (\cite{MannTh}), an element of $\langle D_{1}, D_{2}\rangle$ which is not conjugate to the powers of $D_1, D_{2}$ in $\langle D_{1}, D_{2}\rangle$ is an \textit{atoroidal} fully irreducible.

\end{proof}

\bibliographystyle{alpha}
\bibliography{main}
\end{document}